\documentclass[12pt, reqno]{amsart}
\usepackage{amsmath}  \hoffset=-56pt
\usepackage{amsmath, amsfonts, rawfonts,amssymb,txfonts}
\usepackage{eucal}
\usepackage{latexsym}
\usepackage{cite, color}

\newcommand{\rnnn}{\mathbb R^n}

\newcommand{\sn}{ {\mathbb{S}^{n-1}}}
\newcommand{\R}{\mathbb R}

\newcommand{\psum}{{+_{\negthinspace\kern-2pt p}}\,}
\newcommand{\qsum}[1]{{+_{\negthinspace\kern-2pt #1}}\,}
\newcommand{\dpsum}{{\tilde+_{\negthinspace\kern-1pt p}}\,}
\newcommand{\dqsum}[1]{{\tilde+_{\negthinspace\kern-1pt #1}}\,}
\newcommand{\lsub}[1]{\hskip -1.5pt\lower.5ex\hbox{$_{#1}$}}

\numberwithin{equation}{section}

\newtheorem{theo}{Theorem}[section]
\newtheorem{coro}[theo]{Corollary}
\newtheorem{lem}[theo]{Lemma}

 \theoremstyle{definition}

\begin{document}

\title{On the continuity of the solutions to the $L_{p}$ torsional Minkowski problem}

\author[J. Hu]{Jinrong Hu}
\address{School of Mathematics, Hunan University, Changsha, 410082, Hunan Province, China}
\email{hujinrong@hnu.edu.cn}
\author[Q. Mao]{Qiongfang Mao}
\address{School of Mathematics, Hunan University, Changsha, 410082, Hunan Province, China}
\email{maoqiongfang@hnu.edu.cn}
\author[S. Wang]{Sinan Wang}
\address{School of Mathematics, Hunan University, Changsha, 410082, Hunan Province, China}
\email{wangsinan@hnu.edu.cn}
\begin{abstract}
In this paper, we derive the continuity of solutions to the $L_{p}$ torsional Minkowski problem for $p>1$. It is shown that the weak convergence of the $L_{p}$ torsional measure implies the convergence of the sequence of the corresponding convex bodies in the Hausdorff metric. Furthermore, continuity of the solution to the $L_{p}$ torsional Minkowski problem with regard to $p$ is also obtained.
\end{abstract}

\maketitle

\baselineskip18pt

\parskip3pt

\section{Introduction}

The classic Minkowski problem characterized by the surface area measure is the fundamental problem in the Brunn-Minkowski theory of convex bodies. It was introduced and attacked by Minkowski himself in \cite{M897, M903} and was widely extended in a series of papers \cite{A39,A42,FJ38,B87}.  As an analogue of the classic Minkowski problem within the $L_p$ Brunn-Minkowski theory, the $L_p$ Minkowski problem, was first posed by Lutwak ~\cite{L93}. After that, the $L_{p}$ Minkowski problem have been the breeding ground for many valuable results (see, e.g.,~\cite{B17,CL17,B19,F62,L04,Zhu14,Zhu15,Zh15}). Moreover, with the development of the $L_{p}$ Minkowski problem, the $L_p$ Minkowski problems for other Borel measures were emerged (for example, $p$-capacity, $p$-torsional rigidity). For more details, see \cite{Hu21,J96,J962,C15,Xiong19,Z20} and their references.

In this paper, our aim is to investigate the continuity of the $L_{p}$ Minkowski problem for torsional rigidity in the case of $p>1$. We first recall that the torsional rigidity $T(\Omega)$ of a convex body $\Omega$ in the $n$-dimensional Euclidean space ${\rnnn}$, is defined as
\begin{equation*}\label{tordef}
\frac{1}{T(\Omega)}=\inf\left\{\frac{\int_{\Omega}|\nabla u|^{2}{d}x}{(\int_{\Omega}|u|{d}x)^{2}}: \ u\in W^{1,2}_{0}(\Omega)\, \int_{\Omega}|u|{d}x> 0\right\},
\end{equation*}
where $W^{1,2}_{0}(\Omega)$ denotes the Sobolev space of functions having compact support in $W^{1,2}(\Omega)$, while $W^{1,2}(\Omega)$ represents the Sobolev space of functions having weak derivatives up to first order in $L^{2}(\Omega)$.

It is clear that the explicit expression of torsional functional can be exposed with the assistance of  the solution of an elliptic boundary-value problem (see, e.g., ~\cite{CA05}). More precisely, let $u$ be the unique solution of
\begin{equation}\label{torlapu}
\left\{
\begin{array}{lr}
\Delta u= -2, & x\in \Omega, \\
u=0,  & x\in  \partial \Omega.
\end{array}\right.
\end{equation}
Then we get
\begin{equation*}\label{tordef2}
T(\Omega)=\int_{\Omega}|\nabla u|^{2}{d}x.
\end{equation*}

In the circumstance that the boundary $\partial\Omega$ is of class $C^{2}$, using the standard regularity results of elliptic equations (see, e.g., Gilbarg-Trudinger \cite{GT01}), one see that $\nabla u$ can be suitably defined $\mathcal{H}^{n-1}$ a.e. on $\partial\Omega$, and $T(\Omega)$ can be given as (see, e.g., Proposition 18 of ~\cite{CA05}),
\begin{align}\label{tordef3}
T(\Omega)&=\frac{1}{n+2}\int_{\partial\Omega}h(\Omega,g_{\Omega}(x))|\nabla u(x)|^{2}{d}\mathcal{H}^{n-1}(x)\notag\\
&=\frac{1}{n+2}\int_{\sn}h(\Omega,\theta)|\nabla u(g^{-1}_{\Omega}(\theta)|^{2}{d}S(\Omega,\theta).
\end{align}
From \eqref{tordef3}, we know that the torsion measure $\mu^{tor}(\Omega,\eta)$ is defined on the unit sphere ${\sn}$ by
\begin{equation}\label{tormes2}
\mu^{tor}(\Omega,\eta)=\int_{g^{-1}_{\Omega}(\eta)}|\nabla u(x)|^{2}{d}\mathcal{H}^{n-1}(x)=\int_{\eta}|\nabla u(g^{-1}_{\Omega}(\theta)|^{2}{d}S(\Omega,\theta)
\end{equation}
for every Borel subset $\eta$ of ${\sn}$. Here $h(\Omega,\cdot)$ is the support function of $\Omega$, $\mathcal{H}^{n-1}(\partial \Omega)$, $g_{\Omega}$ and $S(\Omega,\cdot)$ are $(n-1)$-dimensional Hausdorff measure of $\partial\Omega$, the Gauss map and the surface area measure of $\Omega$.

It should be remarked that, from proposition 2.5 of ~\cite{CF10}, one see  that $\nabla u$ has finite non-tangential limits $\mathcal{H}^{n-1}$ a.e. on $\partial \Omega$ via extending the estimates of harmonic functions established by Dahlberg \cite{D77} and $|\nabla u|\in L^{2}(\partial \Omega, \mathcal{H}^{n-1})$ without the assumption of smoothness, which implied that $~\eqref{tormes2}$ is well-defined a.e. on the unit sphere ${\sn}$, not limited to the case of smoothness, and can be deemed as a Borel measure. Furthermore, $~\eqref{tordef3}$ also holds for any convex body in ${\rnnn}$ showed by theorem 3.1 of ~\cite{CF10} and  the Hadamard variational formula of $T(\Omega)$ with the aid of the weak convergence of the torsion measure is given by,
\begin{equation}\label{torhadma}
\frac{d}{dt}T(\Omega+t\Omega_{1})\Big|_{t=0}=\int_{\sn}h(\Omega_{1},\theta){d}\mu^{tor}(\Omega,\theta)
\end{equation}
for any two convex bodies $\Omega, \Omega_{1}$ in ${\rnnn}$. Similar to volume (see, e.g., ~\cite{S14}), Colesanti-Fimiani~\cite{CF10} gave the so-called first mixed torsional rigidity $T_{1}(\Omega, \Omega_{1})$ of $\Omega$ and $\Omega_{1}$ by multiplying $\frac{1}{n+2}$ in the right of the integral in $~\eqref{torhadma}$,

\begin{equation*}\label{mixtorhadma}
T(\Omega, \Omega_{1})=\frac{1}{n+2}\int_{\sn}h(\Omega_{1},\theta){d}\mu^{tor}(\Omega,\theta),
\end{equation*}
which is an extension of $~\eqref{tordef3}$.

Colesanti-Fimiani\cite{CF10} first posed the Minkowski problem for torsional rigidity, which is described as: given $\mu$ is a finite Borel measure on ${\sn}$, what are necessary and sufficient conditions on $\mu$ such that $\mu$ is the torsion measure $\mu^{tor}(\Omega,\cdot)$ of a convex body $\Omega$ in ${\rnnn}$? Colesanti-Fimiani \cite{CF10} proved the existence and uniqueness up to translations of the solution by utilizing the variational technique, which was firstly proposed by Aleksandrov ~\cite{A39, A42} and was later employed by Jerison ~\cite{J96}, Colesanti-Nystr\"{o}m-Salani-Xiao-Yang-Zhang ~\cite{C15}.

Analogue to the $L_{p}$ surface measure (see, e.g., ~\cite{S14}),  Chen-Dai \cite{Chen20} defined the $L_{p}$ torsion measure $\mu^{tor}_{p}(\Omega,\cdot)$ by
\begin{equation}\label{pmea1}
\mu^{tor}_{p}(\Omega,\eta)=\int_{\eta}h(\Omega,\theta)^{1-p}{d}\mu^{tor}(\Omega,\theta)
\end{equation}
for every Borel subset $\eta$ of ${\sn}$, and employed similar methods of Lutwak \cite{L93} and Colesanti-Fimiani \cite{CF10} to get the following $L_{p}$ variational formula of torsional rigidity with respect to $L_{p}$ sum (see, e.g., ~\cite{L04}),
\begin{equation*}\label{ptorhadma}
\frac{d}{dt}T(\Omega+_{p}t\Omega_{1})\Big|_{t=0}=\frac{1}{p}\int_{\sn}h(\Omega_{1},\theta)^{p}{d}\mu^{tor}_{p}(\Omega,\theta)
\end{equation*}
for any two convex bodies $\Omega, \Omega_{1}$ in ${\rnnn}$, where the $L_{p}$ $scalar$ $multiplication$ $t\cdot_{p}\Omega_{1}$ is the set of $t^{1/p}\Omega_{1}$ for $t>0$. Here, we omit the subscript $p$ under the dot.

Meanwhile, Chen-Dai \cite{Chen20} introduced the $L_{p}$ Minkowski problem for torsional rigidity: For $p\in {\R}$, given a finite Borel measure $\mu$ on ${\sn}$, what are the necessary and sufficient conditions on $\mu$ such that $\mu$ is the $L_{p}$ torsion measure $\mu^{tor}_{p}(\Omega,\cdot)$ of a convex body $\Omega$ in ${\rnnn}$? In \cite{Chen20}, the authors proved the existence and uniqueness of the solution when $p>1$. The existence of the solution in the case of $0<p<1$ was treated by Hu-Liu\cite{Hu21}, but the uniqueness of that is not known due to the lack of the $L_{p}$ Brunn-Minkowski inequality for torsional rigidity in the range of $p<1$.

 Notice that there are only few results of the Minkowski type problems on $\mathfrak{p}$-torsion measure with regard to $\mathfrak{p}$-Laplacian equation ($\mathfrak{p}\neq2$), the main reason forming this phenomenon may be that there is no weak convergence result of $\mathfrak{p}$- torsion measure. Huang-Song-Xu \cite{H181} established the Hadamard variational formula of $\mathfrak{p}$-torsional rigidity in the smooth category, further forming a frame for dealing with the associated Minkowski type problems on $\mathfrak{p}$-torsion measure.

We are devoted to deriving the continuity of the solutions to the $L_{p}$ Minkowski problem for torsional rigidity for $p>1$. The main results are shown in the following.
\begin{theo}\label{main}
Let $\mu_{i}$ and $\mu$ be finite Borel measures on $\sn$ which are not concentrated on any closed hemisphere. Suppose that $\Omega_{i}$ and $\Omega$ are convex bodies in $\rnnn$ containing the origin such that $d\mu^{tor}(\Omega_{i},\cdot)=h(\Omega_{i},\cdot)^{p-1}d\mu_{i}$ and $d\mu^{tor}(\Omega,\cdot)=h(\Omega,\cdot)^{p-1}d\mu$ respectively for $1<p<\infty$ and $p\neq n+2$. If $\mu_{i}\rightarrow \mu$ weakly, then $\Omega_{i}\rightarrow \Omega $ as $i\rightarrow \infty$.
\end{theo}
\begin{theo}\label{main2}
Let $\mu$ be finite Borel measures on $\sn$ which are not concentrated on any closed hemisphere. Suppose that $\Omega_{i}$ and $\Omega$ are convex bodies in $\rnnn$ containing the origin such that $d\mu^{tor}(\Omega_{i},\cdot)=h(\Omega_{i},\cdot)^{p_{i}-1}d\mu$ and $d\mu^{tor}(\Omega,\cdot)=h(\Omega,\cdot)^{p-1}d\mu$ respectively for $1<p<\infty$ and $p\neq n+2$. If $p_{i}\rightarrow p$, then $\Omega_{i}\rightarrow \Omega $ as $i\rightarrow \infty$.
\end{theo}
\begin{coro}\label{co}
Let $\mu_{i}$ and $\mu$ be finite Borel measures on $\sn$ which are not concentrated on any closed hemisphere. Suppose that $\Omega_{i}$ and $\Omega$ are convex bodies in $\rnnn$ containing the origin in their interiors such that $d\mu^{tor}_{p}(\Omega_{i},\cdot)=d\mu_{i}$ and $d\mu^{tor}_{p}(\Omega,\cdot)=d\mu$ respectively for $1<p<\infty$ and $p\neq n+2$. If $\mu_{i}\rightarrow \mu$ weakly, then $\Omega_{i}\rightarrow \Omega $ as $i\rightarrow \infty$.
\end{coro}

The organization of this paper goes as follows. In Sec.\ref{Sec2}, we collect some facts about convex bodies and torsional rigidity. In Sec.\ref{Sec3}, we study two related extremal problem associated to the $L_{p}$ Minkowski problem for torsional rigidity. In Sec.\ref{Sec4}, we give the proof of the main theorems of this paper.

\section{Backgrounds}
\label{Sec2}
In this section, we collect some basic facts about convex bodies and torsional rigidity that we shall use in what follows.
\subsection{Basics of convex bodies}

 For standard and good references regarding convex bodies, please refer to Gardner \cite{G06} and Schneider \cite{S14}.

 Let denote by ${\rnnn}$ the $n$-dimensional Euclidean space, and by $o$ the origin of ${\rnnn}$. For $x,y\in {\rnnn}$, $x\cdot y$ denotes the standard inner product.  For $x\in{\rnnn}$, denote by $|x|=\sqrt{x\cdot x}$ the Euclidean norm. The origin-centered ball $B$ is denoted by $\{x\in {\rnnn}:|x|\leq 1\}$, and its boundary by ${\sn}$. A compact convex set of ${\rnnn}$ with non-empty interior is called by a convex body.

If $\Omega$ is a compact convex set in ${\rnnn}$, for $x\in{\rnnn}$, the support function of $\Omega$ is defined by
\[
h(\Omega,x)=\max\{x\cdot y:y \in \Omega\}.
\]

For compact convex sets $\Omega$ and $L$ in ${\rnnn}$, any real $a_{1},a_{2}\geq 0$, define the Minkowski combination of $a_{1}\Omega+a_{2}L$ in ${\rnnn}$ by
\[
a_{1}\Omega+a_{2}L=\{a_{1}x+a_{2}y:x\in \Omega,\ y\in L\},
\]
and its support function is given by
\[
h({a_{1}\Omega+a_{2}L},\cdot)=a_{1}h(\Omega,\cdot)+a_{2}h(L,\cdot).
\]

 The Hausdorff metric $\delta(\Omega,L)$ between two compact convex sets $\Omega$ and $L$ in ${\rnnn}$, is expressed as
\[
\delta(\Omega,L)=\max\{|h(\Omega,\theta)-h(L,\theta)|:\theta\in {\sn}\}.
\]
Let $\Omega_{j}$ be a sequence of compact convex set in ${\rnnn}$. For a compact convex set $\Omega_{0}$ in ${\rnnn}$, if $\delta(\Omega_{j},\Omega_{0})\rightarrow 0$, then $\Omega_{j}$ converges to $\Omega_{0}$, or equivalently, $\sup_{\theta\in{\sn}}|h(\Omega_{j},\theta)-h(\Omega_{0},\theta)|\rightarrow0$, as $j\rightarrow \infty$.

Let $C({\sn})$ be the set of continuous functions defined on the unit sphere ${\sn}$, and $C^{+}({\sn})$ be the set of strictly positive functions in $C({\sn})$. For $f_{1}, f_{2}\in C^{+}({\sn})$, $1\leq p\leq \infty$, the $L_{p}$ sum of $f_{1}$ and $f_{2}$ is defined by
\[
(f_{1}+_{p}t\cdot f_{2})=(f^{p}_{1}+tf^{p}_{2})^{\frac{1}{p}},
\]
for $t> 0$. Here for brevity, we omit the subscript $p$ under the dot.

For any convex body $\Omega$ in ${\rnnn}$ and $\theta\in {\sn}$, the support hyperplane $H(\Omega,\theta)$ in direction $\theta$ is defined by
\[
H(\Omega,\theta)=\{x\in {\rnnn}:x\cdot \theta=h(\Omega,\theta)\},
\]
the half-space $H^{-}(\Omega,\theta)$ in the direction $\theta$ is defined by
\[
H^{-}(\Omega,\theta)=\{x\in {\rnnn}:x\cdot \theta\leq h(\Omega,\theta)\},
\]
and the support set $F(\Omega,\theta)$ in the direction $\theta$ is defined by
\[
F(\Omega,\theta)=\Omega\cap H(\Omega,\theta).
\]

For nonnegative $f\in C(\sn)$, define the Aleksandrov body (also known as Wulff shape) with respect to $f$ by
\[
[f]=\bigcap_{\xi \in \sn}\{x\in \rnnn:x\cdot \xi\leq f(\xi)\}.
\]
Clearly, $[f]$ is a compact convex set containing the origin, and we can see that
\[
h_{[f]}\leq f.
\]

For a compact convex set $\Omega$ in ${\rnnn}$, the diameter of $\Omega$ is denoted by
\[
diam(\Omega)=\max\{|x-y|:x,y \in \Omega\}.
\]

Denote by $\mathcal{P}$ the set of polytopes in ${\rnnn}$, assume that the unit vectors $u_{1},\ldots,u_{N}$ $(N\geq n+1)$ are not concentrated on any closed hemisphere of ${\sn}$, by $\mathcal{P}(u_{1},\ldots,u_{N})$ the set with $P\in \mathcal{P}(u_{1},\ldots,u_{N})$ satisfying
\[
P=\bigcap_{k=1}^{N}H^{-}(P,u_{k}).
\]
 Clearly, we see that, for $P\in \mathcal{P}(u_{1},\ldots,u_{N})$, then $P$ has at most $N$ facets and the outer normals of $P$ are a subset of $\{u_{1},\ldots,u_{N}\}$.

\subsection{Basics of torsional rigidity}

We list some properties of the torsional rigidity and torsion measure (see \cite{CA05,CF10}). Let $\{\Omega_{i}\}_{i=0}^{\infty}$ be a sequence of convex bodies in ${\rnnn}$. On one hand, in terms of torsional rigidity, there are the following results.
\begin{lem}\label{T1}
(i) It is positively homogenous of order $n+2$, i.e.,
\begin{equation*}\label{torhom}
T(m\Omega_{0})=m^{n+2}T(\Omega_{0}),\ m> 0.
\end{equation*}

(ii) It is translation invariant. That is
\begin{equation*}\label{tranin}
T(\Omega_{0} +x_{0})=T(\Omega_{0}),\ \forall x_{0}\in {\rnnn}.
\end{equation*}

(iii) If $\Omega_{i}$ converges to $\Omega_{0}$ in the Hausdorff metric as $i\rightarrow \infty$ (i.e., $\delta(\Omega_{i},\Omega_{0})\rightarrow 0$ as $i\rightarrow \infty$), then
\begin{equation*}\label{Thousdo}
\lim_{i\rightarrow \infty}T(\Omega_{i})=T(\Omega_{0}).
\end{equation*}

(iv) It is monotone increasing, i.e., $T(\Omega_{1})\leq T(\Omega_{2})$ if $\Omega_{1}\subset \Omega_{2}$.
\end{lem}
On the other hand, in terms of torsion measure, there are also the following results.

\begin{lem}\label{meau1}
(a) It is positively homogenous of order $n+1$, i.e.,
\begin{equation*}\label{Tmeahom}
\mu^{tor}(m\Omega_{0},\cdot)=m^{n+1}\mu^{tor}(\Omega_{0},\cdot),\ m> 0.
\end{equation*}

(b) It is translation invariant. That is
\begin{equation*}\label{Tmeatranin}
\mu^{tor}(\Omega_{0} +x_{0})=\mu^{tor}(\Omega_{0}),\ \forall x_{0}\in {\rnnn}.
\end{equation*}

(c) It is absolutely continuous with respect to the surface area measure.

(d) For any fixed $i\in\{0,1,\ldots\}$, if $\Omega_{i}$ converges to $\Omega_{0}$ in the Hausdorff metric as $i\rightarrow \infty$ (i.e., $\delta(\Omega_{i},\Omega_{0})\rightarrow 0$ as $i\rightarrow \infty$), then the sequence of $\mu^{tor}(\Omega_{i},\cdot)$ converges weakly in the sense of measures to $\mu^{tor}(\Omega_{0},\cdot)$ as $i\rightarrow \infty$ .
\end{lem}

In addition, for reader's convenience, we list the following lemmas proved by Colesanti-Fimiani \cite{CF10}, which play a crucial role in this paper.
\begin{lem}\label{argu}Let $\Omega$  be a compact convex set in ${\rnnn}$ and let $u$ be the solution of ~\eqref{torlapu} in $\Omega$, then
\begin{equation*}\label{grdguji}
|\nabla u(x)|\leq diam(\Omega),\ \forall x\in  \Omega.
\end{equation*}
\end{lem}
\begin{lem}\label{argu2}Let $\Omega$, $\Omega_{1}$ be convex bodies in ${\rnnn}$ and let $h(\Omega_{1},\cdot) $ be the support function of $\Omega_{1}$. For sufficiently small $|t|> 0$, there is
\begin{equation*}\label{Thadama6}
\frac{d}{dt}T(\Omega+t\Omega_{1})\big|_{t=0}=\int_{{\sn}}h(\Omega_{1},\theta){d}\mu^{tor}(\Omega,\theta).
\end{equation*}
\end{lem}

We end this section by recalling the important Brunn-Minkowski type inequality and Minkowski inequality of torsional rigidity (see \cite{CA05}).

\begin{lem}\label{BM} Let $\Omega_{0}$, $\Omega_{1}$ be convex bodies in $\rnnn$ and $\lambda\in [0,1]$, the Brunn-Minkowski inequality for torsional rigidity is
\[
T(\lambda\Omega_{0}+(1-\lambda)\Omega_{1})^{1/(n+2)}\geq (1-\lambda)T(\Omega_{0})^{1/(n+2)}+\lambda T(\Omega_{1})^{1/(n+2)},
\]
with equality if and only if $\Omega_{0}$ and $\Omega_{1}$ are homothetic.

\end{lem}

\begin{lem}\label{M} Let $\Omega_{0}$, $\Omega_{1}$ be convex bodies in $\rnnn$, the Minkowski inequality for torsional rigidity is
\[
T(\Omega_{0},\Omega_{1})^{n+2}\geq T(\Omega_{0})^{n+1}T(\Omega_{1}),
\]
with equality if and only if $\Omega_{0}$ and $\Omega_{1}$ are homothetic.

\end{lem}
\section{Two related extreme problems for $p$-torsional rigidity}
\label{Sec3}
In this section, inspired by \cite{Z20}, we study two related extreme problems for the existence of the $L_{p}$ Minkowski problem for torsional rigidity.

Let $1<p<\infty$. Suppose $\mu$ is a finite Borel measure on $\sn$ which is not concentrated on any closed hemisphere. For a compact convex set $Q$ containing the origin, set the functional
\[
F_{p}(Q)=\frac{1}{n+2}\int_{\sn}h(Q,\cdot)^{p}d\mu.
\]

Two problems are showed in the followings.

{\bf Problem 1.} Among all convex bodies $Q$ in $\rnnn$ containing the origin, seek one to deal with the constrained minimization problem
\[
\inf_{Q}F_{p}(Q) \quad \quad subject \ to \quad T(Q)\geq 1.
\]

{\bf Problem 2.} Among all convex bodies $Q$ in $\rnnn$ containing the origin, seek one to deal with the following constrained maximum problem
\[
\sup_{Q}T(Q) \quad \quad subject \ to \quad F_{p}(Q)\leq 1.
\]
Here we give some remarks. Let $Q$ be a convex body in $\rnnn$ containing the origin. From \cite[Lemma 2.3]{Z20}, one see
\begin{equation}\label{}
0<F_{p}(Q)<\infty.
\end{equation}

By the homogeneity of $F_{p}$ and $T$, one see $F_{p}(mQ)=m^{p}F_{p}(Q)$ and $T(mQ)=m^{n+2}T(Q)$ for $m>0$, it suffices to conclude that, if $Q$ solves Problem 1, then $T(Q)=1$. Similarly, if $Q$ solves Problem 2, then $F_{p}(Q)=1$.

The following lemma shows the relations between Problems 1 and 2.
\begin{lem}\label{re12}
For Problems 1 and 2, the following assertions hold.
\begin{itemize}

    \item[(1)] If convex body $\bar{\Omega}$ solves Problem 1, then
    \[
    \Omega:=F_{p}(\bar{\Omega})^{-1/p}\bar{\Omega}
    \]
    solves Problem 2.
    \item[(2)] If convex body $\Omega$ solves Problem 2, then
    \[
 \bar{\Omega}:=T(\Omega)^{-1/(n+2)}\Omega
    \]
solves Problem 1.
    \end{itemize}
\end{lem}
\begin{proof}
For (1). Assume $\bar{\Omega}$ solves Problem 1. Let $Q$ be a convex body containing the origin such that $F_{p}(Q)\leq1$. With the aid of the positive homogeneity of $T$, $T(\bar{\Omega})=1$, $T(T(Q)^{-1/(n+2)}Q)=1$, the positive homogeneity of $F_{p}$, and $F_{p}(Q)\leq1$, we have
\begin{equation}
\begin{split}
\label{Up3}
T(\Omega)^{p/(n+2)}&=F_{p}(\bar{\Omega})^{-1}\\
&\geq F_{p}(T(Q)^{-1/(n+2)}Q)^{-1}\\
&=F_{p}(Q)^{-1}T(Q)^{p/(n+2)}\\
&\geq T(Q)^{p/(n+2)}.
\end{split}
\end{equation}
Thus $K$ solves Problem 2.

For (2). Assume $\Omega$ solves Problem 2. Let $Q$ be convex body containing the origin such that $T(Q)\geq 1$. By the positive homogeneity of $F_{p}$, $F_{p}(\Omega)=1$, $F_{p}(F_{p}(Q)^{-1/p}Q)=1$, the positive homogeneity of $T$, and $T(Q)\geq 1$, we obtain
\begin{equation}
\begin{split}
\label{Up3}
F_{p}(\bar{\Omega})&=T(\Omega)^{-p/(n+2)}\\
&\leq T(F_{p}(Q)^{-1/p}Q)^{-p/(n+2)}\\
&=F_{p}(Q)T(Q)^{-p/(n+2)}\\
&\leq F_{P}(Q).
\end{split}
\end{equation}
Thus, $\bar{\Omega}$ solves Problem 1.
\end{proof}
The following is the normalized $L_{p}$ Minkowski problem for torsional rigidity.

{\bf Problem 3.} Among all convex bodies in $\rnnn$ containing the origin, find a convex body $\Omega$ such that
\[
\frac{d\mu^{tor}(\Omega,\cdot)}{T(\Omega)}=h(\Omega,\cdot)^{p-1}d\mu.
\]
The following lemma digs up relations between Problems 2 and 3.
\begin{lem}\label{LP23}
Suppose $\mu$ is a finite Borel measure on $\sn$ and is not concentrated on any closed hemisphere. Let $1<p<\infty$, and $\Omega$ be a convex body in $\rnnn$ containing the origin. Then the following assertions hold.
\begin{itemize}

    \item [(1)] If $\Omega$ contains the origin in its interior and solves Problem 2 for $(\mu,p)$, then it solves Problem 3 for $(\mu,p)$.
    \item [(2)] If $\Omega$ solves Problem 3 for $(\mu, p)$, then it solves Problem 2 for $(\mu,p)$.
    \end{itemize}
\end{lem}
\begin{proof}
Let $\Omega$ be a convex body containing the origin in its interior. Assume $\Omega$ is the solution to Problem 2. We shall prove that it also attack Problem 3. Now, taking a  $f\in C(\sn)$, let $t\in (-t_{0},t_{0})$,
\[
\Omega_{t}=[h(\Omega,\cdot)+tf] \quad and \quad F_{p}(h(\Omega,\cdot)+tf)=\frac{1}{n+2}\int_{\sn}(h(\Omega,\cdot)+tf)^{p}d\mu,
\]
where $t_{0}>0$ is chosen so that $h(\Omega,\cdot)+tf>0$ and $F_{p}(h(\Omega,\cdot)+tf)=1$. Since $\Omega$ solves Problem 2, this implies that there exists a constant $b^{'}$ such that
\[
\frac{d T(\Omega_{t})}{dt}\Big|_{t=0}=b^{'}\frac{d F_{p}(h(\Omega,\cdot)+tf)}{dt}\Big|_{t=0}.
\]
Then
\[
\int_{\sn}f(\theta) d\mu^{tor}(\Omega,\theta)=b^{'}\frac{p}{n+2}\int_{\sn}f(\theta)h(\Omega,\theta)^{p-1}d\mu.
\]
Since $f$ is arbitrary, this allows us to conclude that $\mu^{tor}(\Omega,\cdot)=b^{'}\frac{p}{n+2}h(\Omega,\cdot)^{p-1}\mu$. Due to $F_{p}(\Omega)=1$, it yields $b^{'}=\frac{n+2}{p}T(\Omega)$. Thus $T(\Omega)^{-1}d\mu^{tor}(\Omega,\cdot)=h(\Omega,\cdot)^{p-1}d\mu$. So, $\Omega$ solves Problem 3.

For proving (2). Assume $\Omega$ solves Problem 3. Let $Q$ be a convex body containing the origin, such that $1=\frac{1}{n+2}\int_{\sn}h(Q,\cdot)^{p}d\mu$. We need to prove $T(\Omega)\geq T(Q)$.

Due to $T(\Omega)h(\Omega,\cdot)^{p-1}d\mu=d\mu^{tor}(\Omega,\cdot)$, then
\begin{equation}
\begin{split}
\label{Up3}
1&=\frac{1}{n+2}\int_{\{h(Q,\cdot)>0\}}h(Q,\cdot)^{p}d\mu+\frac{1}{n+2}\int_{\{h(Q,\cdot)=0\}}h(Q,\cdot)^{p}d\mu\\
&\geq \frac{1}{n+2}\int_{\{h(Q,\cdot)>0\}}h(Q,\cdot)^{p}d\mu\\
&=\frac{1}{n+2}\int_{\{h(Q,\cdot)>0\}}\left( \frac{h(Q,\cdot)}{h(\Omega,\cdot)} \right)^{p}\frac{h(\Omega,\cdot)}{T(\Omega)}d\mu^{tor}(\Omega,\cdot).
\end{split}
\end{equation}
Since the measure $\frac{1}{(n+2)T(\Omega)}h(\Omega,\cdot)d\mu^{tor}(\Omega,\cdot)$ is a Borel probability measure on $\{h(\Omega,\cdot)\neq 0\}$. Then, by the Jensen inequality, the definition of mixed torsional rigidity, and the Minkowski inequality for torsional rigidity, we have
\begin{equation}
\begin{split}
\label{Up3}
1&\geq \left( \frac{1}{(n+2)T(\Omega)} \int_{\{ h(\Omega,\cdot)>0\}}\left( \frac{h(Q,\cdot)}{h(\Omega,\cdot)} \right)^{p}h(\Omega,\cdot)d\mu^{tor}(\Omega,\cdot)\right)^{1/p}\\
&\geq \frac{1}{(n+2)T(\Omega)} \int_{\{ h(\Omega,\cdot)>0\}}\frac{h(Q,\cdot)}{h(\Omega,\cdot)}h(\Omega,\cdot)d\mu^{tor}(\Omega,\cdot)\\
&=\frac{1}{(n+2)T(\Omega)}\int_{\{ h(\Omega,\cdot)>0\}}h(Q,\cdot)d\mu^{tor}(\Omega,\cdot)\\
&=\frac{T(\Omega,Q)}{T(\Omega)}-\frac{1}{(n+2)T(\Omega)}\int_{\{ h(\Omega,\cdot)=0\}}h(Q,\cdot)d\mu^{tor}(\Omega,\cdot)\\
&\geq \left(\frac{T(Q)}{T(\Omega)} \right)^{\frac{1}{n+2}}-\frac{1}{(n+2)T(\Omega)}\int_{\{ h(\Omega,\cdot)=0\}}h(Q,\cdot)d\mu^{tor}(\Omega,\cdot).
\end{split}
\end{equation}
Since $T(\Omega)h(\Omega,\cdot)^{p-1}d\mu=d \mu^{tor}(\Omega,\cdot)$, one see that
\[
\int_{\{ h(\Omega,\cdot)=0\}}h(Q,\cdot)d\mu^{tor}(\Omega,\cdot)=T(\Omega)\int_{\{ h(\Omega,\cdot)=0\}}h(Q,\cdot)h(\Omega,\cdot)^{p-1}d\mu=0.
\]
So,
\[
1\geq \left(\frac{T(Q)}{T(\Omega)} \right)^{\frac{1}{n+2}}.
\]
We conclude that $T(\Omega)\geq T(Q)$.
\end{proof}
The following lemma tells us that the uniqueness of solution to Problem 2. Equivalently, the uniqueness of solution to problem 1 is revealed.

\begin{lem}\label{LP33}
Suppose $\mu$ is  a finite Borel measure on $\sn$ and is not concentrated on any closed hemisphere. Let $1<p<\infty$. If $\Omega_{1}$ and $\Omega_{2}$ are convex bodies in $\rnnn$ containing the origin and solving Problem 2 for $(\mu,p)$, then $\Omega_{1}=\Omega_{2}$.
\end{lem}
\begin{proof}
Consider the convex body $2^{-1}(\Omega_{1}+\Omega_{2})$. It also contains the origin. Since
\[
F_{p}\left(\frac{\Omega_{1}+\Omega_{2}}{2}\right)\leq \frac{F_{p}(\Omega_{1})+F_{p}(\Omega_{2})}{2}\leq 1,
\]
then the convex body $2^{-1}(\Omega_{1}+\Omega_{2})$ still satisfies the constraint conditions in Problem 2. This fact illustrates
\[
T\left(\frac{\Omega_{1}+\Omega_{2}}{2}\right)\leq T(\Omega_{1})=T(\Omega_{2}).
\]
As a result,
\begin{equation}\label{TI}
T\left(\frac{\Omega_{1}+\Omega_{2}}{2} \right)^{1/(n+2)}\leq \frac{1}{2}T(\Omega_{1})^{1/(n+2)}+\frac{1}{2}T(\Omega_{2})^{1/(n+2)}.
\end{equation}
On the other hand, the reverse of $\eqref{TI}$ holds, so equality occurs in \eqref{TI},  by lemma \ref{BM}, it yields $\Omega_{1}=\alpha \Omega_{2}+x$, for some $\alpha>0$ and $x\in \rnnn$. By $T(\Omega_{1})=T(\Omega_{2})$ together with the homogeneity and translation invariance of $T$, we get $\alpha=1$. Then $\Omega_{1}=\Omega_{2}+x$. Next, we claim $x=o$.

Since
\[
T\left( \frac{\Omega_{1}+\Omega_{2}}{2}\right)=T(\Omega_{1})=T(\Omega_{2}),
\]
this implies that $2^{-1}(\Omega_{1}+\Omega_{2})$ is also a solution to problem 2. So $F_{p}(2^{-1}(\Omega_{1}+\Omega_{2}))=1$, then
\[
\int_{\sn}\left( \frac{h(\Omega_{1},\cdot)+h(\Omega_{2},\cdot)}{2}\right)^{p}d\mu=\int_{\sn}\frac{h(\Omega_{1},\cdot)^{p}+h(\Omega_{2},\cdot)^{p}}{2}d\mu.
\]
In view of the fact that
\[
\left(\frac{h(\Omega_{1},\cdot)+h(\Omega_{2},\cdot)}{2}\right)^{p}\leq\frac{h(\Omega_{1},\cdot)^{p}+h(\Omega_{2},\cdot)^{p}}{2}.
\]
Thus, for all $\xi\in supp \mu$, there is
\begin{equation}\label{hp}
\frac{h(\Omega_{2},\xi)+(h(\Omega_{2},\xi)+(\xi\cdot x))}{2}=\left(\frac{h(\Omega_{2},\xi)^{p}+(h(\Omega_{2},\xi)+(\xi\cdot x))^{p}}{p}\right)^{1/p}.
\end{equation}
Let $\mathcal{O}=\{\xi \in \sn : x\cdot \xi>0\}$. Since $\mu$ is not concentrated on any closed hemisphere, it yields that $\mu(\mathcal{O}\cap supp \mu)>0$. If $x$ is nonzero, then for any $\xi\in \mathcal{O} \cap supp \mu$,
\begin{equation}\label{hp2}
\frac{h(\Omega_{2},\xi)+(h(\Omega_{2},\xi)+(\xi\cdot x))}{2}<\left(\frac{h(\Omega_{2},\xi)^{p}+(h(\Omega_{2},\xi)+(\xi\cdot x))^{p}}{p}\right)^{1/p},
\end{equation}
which contradicts \eqref{hp}. Hence, we verify $x=o$, which tell us  $\Omega_{1}=\Omega_{2}$.

\end{proof}

The original $L_{p}$ Minkowski problem for torsional rigidity is given as follows.

{\bf Problem 4.} Among all convex bodies in $\rnnn$ that contain the origin, find a convex body $\Omega$ such that
\[
d \mu^{tor}(\Omega,\cdot)=h(\Omega,\cdot)^{p-1}d\mu.
\]

The following lemma shows the equivalence between Problems 4 and 3.
\begin{lem}\label{r34}
Suppose $\Omega$ is a convex body in $\rnnn$ containing the origin and that $p\neq n+2$. Then the following conclusions hold.
\end{lem}
\begin{itemize}

    \item [(1)] If $\Omega$ solves problem 3, then $\bar{\Omega}=T(\Omega)^{1/(p-2-n)}\Omega$ solves Problem 4.

        \item [(2)] If $\Omega$ solves Problem 4, then $\tilde{\Omega}=T(\Omega)^{-1/p}\Omega$ solves Problem 3.

    \end{itemize}

\begin{proof}
Assume that $\Omega$ solves Problem 3. By the homogeneity of $\mu^{tor}$, and $\frac{d\mu^{tor}(\Omega,\cdot)}{T(\Omega)}=h(\Omega,\cdot)^{p-1}d\mu$, it follows that
\begin{equation}
\begin{split}
\label{Up3}
d\mu^{tor}(\bar{\Omega},\cdot)&=d\mu^{tor}(T(\Omega)^{1/(p-2-n)}\Omega,\cdot)\\
&=T(\Omega)^{(n+1)/(p-2-n)}d\mu^{tor}(\Omega,\cdot)\\
&=T(\Omega)^{(n+1)/(p-2-n)}T(\Omega)h(\Omega,\cdot)^{p-1}d\mu\\
&=T(\Omega)^{(n+1)/(p-2-n)}T(\Omega)h(T(\Omega)^{1/(n+2-p)}\bar{\Omega},\cdot)^{p-1}d\mu\\
&=h(\bar{\Omega},\cdot)^{p-1}d\mu.
\end{split}
\end{equation}
(2) Assume that $\Omega$ solves Problem 4. From $d\mu^{tor}(\Omega,\cdot)=h(\Omega,\cdot)^{p-1}d\mu$, we have
\begin{equation}
\begin{split}
\label{Up3}
\frac{d\mu^{tor}(\tilde{\Omega},\cdot)}{T(\tilde{\Omega})}&=\frac{d\mu^{tor}(T(\Omega)^{-1/p}\Omega,\cdot)}{T(T(\Omega)^{-1/p}\Omega)}\\
&=\frac{T(\Omega)^{-(n+1)/p}d\mu^{tor}(\Omega,\cdot)}{T(\Omega)^{-(n+2)/p}T(\Omega)}\\
&=T(\Omega)^{(1-p)/p}h(\Omega,\cdot)^{p-1}d\mu\\
&=T(\Omega)^{(1-p)/p}h(T(\Omega)^{1/p}\tilde{\Omega})^{p-1}d\mu\\
&=h(\tilde{\Omega},\cdot)^{p-1}d\mu.
\end{split}
\end{equation}
Hence, we complete the proof.
\end{proof}

By means of the similar approximation arguments in \cite{Hg05}, we can obtain the solvability of Problem 3 showed in the followings.
\begin{theo}\label{TM}
Let $\mu$ be a finite Borel measure on $\sn$ which is not concentrated on any closed hemisphere, and $1<p<\infty$. Then there exists a unique convex body $\Omega$ containing the origin, such that $\frac{d\mu^{tor}(\Omega,\cdot)}{T(\Omega)}=h(\Omega,\cdot)^{p-1}d\mu$.
\end{theo}
For reader's convenience, we give the proof of Theorem \ref{TM} in the appendix.

\section{Proof of the main theorems}
\label{Sec4}
In this section, we give the proof of the main theorems beginning with doing some preparation. Let $1<p<\infty$ and $p\neq n+2$.

Let $\mu$ be a finite Borel measure on $\sn$ which is not concentrated on any closed hemisphere, and $\Omega$ be a convex body in $\rnnn$ containing the origin. Define the functional $F_{p}(\Omega)$ by
\[
F_{p}(\Omega)=\frac{1}{n+2}\int_{\sn}h(\Omega,\cdot)^{p}d\mu.
\]

\begin{lem}\label{hj}
Let $\mu$ be a finite Borel measure on $\sn$ which is not concentrated on any closed hemisphere and $Q$ be a convex body in $\rnnn$ containing the origin with $F_{p}(Q)\leq 1$. Suppose that convex body $\Omega$ is the solution to Problem 3 for $(\mu,p)$. Then $F_{p}(\Omega)=1$ and $T(\Omega)\geq T(Q)$.

\end{lem}
\begin{proof}
Since $\Omega$ is a solution to Problem 3, it follows that $\frac{d\mu^{tor}(\Omega,\cdot)}{T(\Omega)}=h(\Omega,\cdot)^{p-1}d\mu$. We obtain
\begin{equation}
\begin{split}
\label{Up3}
F_{p}(\Omega)&=\frac{1}{n+2}\int_{\sn}h(\Omega,v)^{p}d\mu(v)\\
&=\frac{1}{n+2}\int_{\{ h(\Omega,\cdot)>0 \}}h(\Omega,v)^{p}d\mu(v)\\
&=\frac{1}{n+2}\int_{\{ h(\Omega,\cdot)>0 \}}h(\Omega,v)^{p}h(\Omega,v)^{1-p}\frac{1}{T(\Omega)}d\mu^{tor}(\Omega,v)\\
&=1,
\end{split}
\end{equation}
where the measure $\frac{h(\Omega,\cdot)}{(n+2)T(\Omega)}d\mu^{tor}(\Omega,\cdot)$ is a Borel probability measure on $\{h(\Omega,\cdot)>0\}$. By $F_{p}(Q)\leq 1$ and $\frac{d\mu^{tor}(\Omega,\cdot)}{T(\Omega)}=h(\Omega,\cdot)^{p-1}d\mu$, the Jensen inequality, the definition of the mixed torsional rigidity, the torsional Minkowski inequality, we get
\begin{equation}
\begin{split}
\label{Up3}
1&\geq F_{p}(Q)^{1/p}=\left(\frac{1}{n+2}\int_{\sn}h(Q,v)^{p}d\mu(v) \right)^{1/p}\\
&=\left(  \frac{1}{n+2}\int_{\{h(\Omega,\cdot)>0\}}h(Q,v)^{p}d\mu(v)+ \frac{1}{n+2}\int_{\{h(\Omega,\cdot)=0\}}h(Q,v)^{p}d\mu(v)\right)^{1/p}\\
&\geq \left( \frac{1}{n+2} \int_{\{h(\Omega,\cdot)>0\}}h(Q,v)^{p}d\mu(v) \right)^{1/p}\\
&=\left(\frac{1}{n+2}\int_{\{h(\Omega,\cdot)>0\}}\left(\frac{h(Q,v)}{h(\Omega,v)} \right)^{p}  \frac{h(\Omega,v)}{T(\Omega)}d\mu^{tor}(\Omega,v) \right)^{1/p}\\
&\geq \frac{1}{(n+2)T(\Omega)}\int_{\{h(\Omega,\cdot)>0\}}h(Q,v)d\mu^{tor}(\Omega,v)\\
&=\frac{T(\Omega,Q)}{T(\Omega)}-\frac{1}{n+2}\int_{\{h(\Omega,\cdot)=0\}}h(Q,v)h(\Omega,v)^{p-1}d\mu(v)\\
&\geq \left(\frac{T(Q)}{T(\Omega)} \right)^{\frac{1}{n+2}}.
\end{split}
\end{equation}
Hence, we obtain $T(\Omega)\geq T(Q)$.
\end{proof}
\begin{lem}\label{ui}
Let $\mu_{i}$ and $\mu$ be finite Borel measures on $\sn$ which are not concentrated on any closed hemisphere. Suppose that convex bodies $\Omega_{i}$ and $\Omega$ in $\rnnn$ containing the origin are solutions to Problem 3 for $(\mu_{i},p)$ and $(\mu,p)$, respectively. If $\mu_{i}\rightarrow \mu$ weakly, then $\Omega_{i}\rightarrow \Omega$ as $i\rightarrow \infty$.

\end{lem}

\begin{proof}
First, we prove the sequence $\{\Omega_{i}\}$ is bounded. Let
\[
R_{i}=h(\Omega_{i},u_{i})=\max_{\sn}h(\Omega_{i},u).
\]
Due to $R_{i}u_{i}\in \Omega_{i}$, $h(\Omega_{i},v)\geq R_{i} \max\{u_{i}\cdot v,0\}=R_{i}(u_{i}\cdot v)_{+}$ for $v\in \sn$. By Lemma 4.1, we have
\begin{equation}\label{RI}
n+2=\int_{\sn}h(\Omega_{i},v)^{p}d\mu_{i}(v)\geq R^{p}_{i}\int_{\sn}(u_{i}\cdot v)^{p}_{+}d\mu_{i}(v).
\end{equation}
Since $\mu$ is not concentrated on any closed hemisphere, it follows that for any $u\in \sn$, $\int_{\sn}(u\cdot v)^{p}_{+}d\mu(v)>0$. Using the continuity of $u\mapsto \int_{\sn}(u\cdot v)^{p}_{+}d\mu(v)$ on $\sn$, there exists a constant $C>0$ depending on $p$, such that for any $u\in \sn$, $\int_{\sn}(u\cdot v)^{p}_{+}d\mu (v)\geq C>0$. By $\mu_{i}\rightarrow \mu$ weakly, there exists an $N>0$ such that
for $i\geq N$ and $u\in \sn$,
\[
\int_{\sn}(u\cdot v)^{p}_{+}d\mu_{i}(v)\geq \frac{C}{2}>0.
\]
Applying \eqref{RI}, we get
\[
R_{i}\leq \left( \frac{2(n+2)}{C}\right)^{1/p}.
\]
So, $\{\Omega_{i}\}$ is bounded. Second, by the Blaschke selection theorem, there exists a subsequence of $\{\Omega _{i}\}$, still denoted by $\{\Omega_{i}\}$ such that $\Omega_{i}\rightarrow \hat{\Omega}$. On one hand, if $Q=\lambda_{i}B$ with $\lambda_{i}=(\frac{1}{n+2}|\mu_{i}|)^{-\frac{1}{p}}$, then $1=\frac{1}{n+2}\int_{\sn}h(Q,v)^{p}d\mu_{i}(v)$. By Lemma 4.1, $T(\Omega_{i})\geq T(Q)$. Due to $\mu_{i}\rightarrow \mu$ weakly and $0<|\mu|< \infty$, it follows that
$0<\sup_{i}\mu_{i}<\infty$. Hence,
\[
T(\Omega_{i})\geq T(\lambda_{i}B)\geq T(B)\left(\frac{1}{n+2} \sup_{i}|\mu_{i}|\right)^{-\frac{n+2}{p}}>0.
\]
By lemma \ref{argu}, we have $Vol(\Omega_{i})\geq \frac{1}{diam(\Omega_{i})^{2}}T(\Omega_{i})\geq \frac{1}{d^{2}}T(B)\left(\frac{1}{n+2} \sup_{i}|\mu_{i}|\right)^{-\frac{n+2}{p}}>0$ for some $d>0$, thus, using the continuity of volume, we get
\[
Vol(\hat{\Omega})=\lim_{i\rightarrow \infty}Vol(\Omega_{i})>0.
\]
Hence, $\hat{\Omega}$ is a convex body.

We are in position to show that $\hat{\Omega}=\Omega$. Since $\Omega_{i}\rightarrow \hat{\Omega}$, $h(\Omega_{i},\cdot)\rightarrow h(\hat{\Omega},\cdot)$ uniformly on $\sn$, $T(\Omega_{i})\rightarrow T(\hat{\Omega})$ and $\mu^{tor}(\Omega_{i},\cdot)\rightarrow \mu^{tor}(\hat{\Omega},\cdot)$ weakly. Together with $\mu_{i}\rightarrow \mu$ weakly and $\frac{d\mu^{tor}(\Omega_{i},\cdot)}{T(\Omega_{i})}=h(\Omega_{i},\cdot)^{p-1}d\mu_{i}$, we have
\[
\frac{d\mu^{tor}(\hat{\Omega},\cdot)}{T(\hat{\Omega})}=h(\hat{\Omega},\cdot)^{p-1}d\mu.
\]
By Theorem \ref{TM}, it implies that $\hat{\Omega}=\Omega$. So, $\Omega_{i}\rightarrow \Omega$.
\end{proof}
\begin{lem}\label{pi}
Let $\mu$ be a finite Borel measure on $\sn$ which is not concentrated on any closed hemisphere. Suppose that convex bodies $\Omega_{i}$ and $\Omega$ in $\rnnn$  containing the origin are solutions to Problem 3 for $(\mu,p_{i})$ and $(\mu,p)$ respectively. If $p_{i}\rightarrow p$, then $\Omega_{i}\rightarrow \Omega$ as $i\rightarrow\infty$.
\end{lem}
\begin{proof}
First, we prove that the sequence $\{\Omega_{i}\}$ is bounded. By lemma 4.1,
\[
1=\frac{1}{n+2}\int_{\sn}h(\Omega_{i},v)^{p_{i}}d\mu(v).
\]
Let
\[
R_{i}=h(\Omega_{i},u_{i})=\max_{u\in\sn}h(\Omega_{i},u).
\]
Then, $R_{i}u_{i}\in \Omega$ and therefore for $v\in \sn$,
\[
h(\Omega_{i},v)\geq R_{i}(u_{i}\cdot v)_{+}.
\]
Third, since $\mu$ is not concentrated on any closed hemisphere, there exists a constant $c_{0}>0$ such that for $u\in \sn$,
\begin{equation}\label{}
\int_{\sn}(u\cdot v)^{p+\varepsilon}_{+}d\mu(v)\geq (n+2)\frac{1}{c^{p+\varepsilon}_{0}}>0.
\end{equation}
Due to $p_{i}\rightarrow p$ and $p>1$, for $\varepsilon>0$ and sufficiently large $i$, $1<p-\varepsilon<p_{i}<p+\varepsilon$ and there exists an $a_{0}>0$ such that $c^{p+\varepsilon}_{0}<(c_{0}+a_{0})^{p_{i}}$,
\begin{equation}
\begin{split}
\label{Up3}
1&=\frac{1}{n+2}\int_{\sn}h(\Omega_{i},v)^{p_{i}}d\mu(v)\\
&\geq \frac{1}{n+2}\int_{\sn}R^{p_{i}}_{i}(u_{i}\cdot v)^{p_{i}}_{+}d\mu (v)\\
&\geq \frac{1}{n+2}R^{p_{i}}_{i}\int_{\sn}(u_{i}\cdot v)^{p+\varepsilon}_{+}d\mu (v)\\
&\geq R^{p_{i}}_{i}\frac{1}{c^{p+\varepsilon}_{0}}\\
&\geq \left(\frac{R_{i}}{c_{0}+a_{0}}\right)^{p_{i}}.
\end{split}
\end{equation}
Because $p_{i}\geq1$, it follows that $R_{i}\leq c_{0}+a_{0}$, hence, the sequence $\{\Omega_{i}\}$ is bounded. Second, by the Blaschke selection theorem, there exists a subsequence of $\{\Omega _{i}\}$, still denoted by $\{\Omega_{i}\}$ such that $\Omega_{i}\rightarrow \hat{\Omega}$. On one hand, if $Q=\lambda_{i}B$ with $\lambda_{i}=(\frac{1}{n+2}|\mu|)^{-\frac{1}{p_{i}}}$, then $1=\frac{1}{n+2}\int_{\sn}h(Q,v)^{p_{i}}d\mu(v)$. By Lemma 4.1, $T(\Omega_{i})\geq T(Q)$. Due to $p_{i}\rightarrow p$ and $p>1$, it follows that for sufficiently large $i$,

\[
T(\Omega_{i})\geq T(\lambda_{i}B)\geq \frac{1}{2}T(B)\left(\frac{1}{n+2}|\mu| \right)^{-\frac{n+2}{p}}>0.
\]
From lemma \ref{argu}, we know that $Vol(\Omega_{i})\geq \frac{1}{diam(\Omega_{i})^{2}}T(\Omega_{i})\geq \frac{1}{2d^{2}}T(B)\left(\frac{1}{n+2}|\mu| \right)^{-\frac{n+2}{p}}>0$ holds for some $d>0$,  it follows that
\[
Vol(\hat{\Omega})=\lim_{i\rightarrow \infty}Vol(\Omega_{i})>0.
\]
Hence, $\hat{\Omega}$ is a convex body.

We are in position to show that $\hat{\Omega}=\Omega$. Since $\Omega_{i}\rightarrow \hat{\Omega}$, then $h(\Omega_{i},\cdot)\rightarrow h(\hat{\Omega},\cdot)$ uniformly on $\sn$, $T(\Omega_{i})\rightarrow T(\hat{\Omega})$ and $\mu^{tor}(\Omega_{i},\cdot)\rightarrow \mu^{tor}(\hat{\Omega},\cdot)$ weakly. Together with $p_{i}\rightarrow p$ and $\frac{d\mu^{tor}(\Omega_{i},\cdot)}{T(\Omega_{i})}=h(\Omega_{i},\cdot)^{p_{i}-1}d\mu$, we have
\[
\frac{d\mu^{tor}(\hat{\Omega},\cdot)}{T(\hat{\Omega})}=h(\hat{\Omega},\cdot)^{p-1}d\mu.
\]
By Theorem \ref{TM}, it implies that $\hat{\Omega}=\Omega$. So, $\Omega_{i}\rightarrow \Omega$.

Now, we are in place to prove the main results.

\emph{Proof of Theorem 1.1.} Let convex bodies $\Omega_{i}$ and $\Omega$ in $\rnnn$ containing the origin satisfy $d\mu^{tor}(\Omega_{i},\cdot)=h(\Omega_{i},\cdot)^{p-1}d\mu_{i}$ and $d\mu(\Omega,\cdot)=h(\Omega,\cdot)^{p-1}d\mu$, which $\Omega_{i}$ and $\Omega$ are solutions to Problem 4 for $(\mu_{i},p)$ and $(\mu,p)$, respectively. By Lemma \ref{r34}(2), it follows that $\tilde{\Omega}_{i}=T(\Omega_{i})^{-1/p}\Omega_{i}$, $\tilde{\Omega}=T(\Omega)^{-1/p}(\Omega)$ are solutions to Problem 3 for $(\mu_{i},p)$ and $(\mu,p)$ respectively. Since $\mu_{i}\rightarrow \mu$ weakly, by Lemma \ref{ui}, it follows that $\tilde{\Omega}_{i}\rightarrow \tilde{\Omega}$ as $i\rightarrow \infty$. By Lemma \ref{r34}(1) and the continuity of torsional rigidity, it follows that
\[
\Omega_{i}=T(\tilde{\Omega}_{i})^{1/(p-2-n)}\tilde{\Omega}_{i}\rightarrow  T(\tilde{\Omega})^{1/(p-2-n)}\tilde{\Omega}=\Omega.
\]
 This completes the proof.

\emph{Proof of Theorem 1.2.} Let convex bodies $\Omega_{i}$ and $\Omega$ in $\rnnn$ containing the origin satisfy $d\mu^{tor}(\Omega_{i},\cdot)=h(\Omega_{i},\cdot)^{p_{i}-1}d\mu$ and $d\mu(\Omega,\cdot)=h(\Omega,\cdot)^{p-1}d\mu$, which $\Omega_{i}$ and $\Omega$ are solutions to Problem 4 for $(\mu,p_{i})$ and $(\mu,p)$, respectively. By Lemma \ref{r34}(2), it follows that $\tilde{\Omega}_{i}=T(\Omega_{i})^{-1/p_{i}}\Omega_{i}$, $\tilde{\Omega}=T(\Omega)^{-1/p}(\Omega)$ are solutions to Problem 3 for $(\mu,p_{i})$ and $(\mu,p)$ respectively. Since $p_{i}\rightarrow p$, by Lemma \ref{pi}, it follows that $\tilde{\Omega}_{i}\rightarrow \tilde{\Omega}$ as $i\rightarrow \infty$, by Lemma \ref{r34}(1) and the continuity of torsional rigidity, it follows that
\[
\Omega_{i}=T(\tilde{\Omega}_{i})^{1/(p_{i}-2-n)}\tilde{\Omega}_{i}\rightarrow  T(\tilde{\Omega})^{1/(p-2-n)}\tilde{\Omega}=\Omega.
\]
 This completes the proof.

\end{proof}

\section{Appendix-proof of Theorem \ref{TM}}
\label{Sec5}
We first prove the discrete version of Theorem \ref{TM} as follows.
\begin{theo}\label{AD}
Suppose $\mu$ is a discrete measure on $\sn$ that is not concentrated on any closed hemisphere. Then, for each $1<p<\infty$, there exists a unique polytope $P$ containing the origin in its interior such that
\[
\frac{\mu^{tor}_{p}(P,\cdot)}{T(P)}=\mu.
\]
\end{theo}
To prove Theorem \ref{AD}, we first do some preparation.
\begin{lem}\label{ra}
Suppose $\mu$ is a discrete measure on $\sn$ and is not concentrated on any closed hemisphere. Let $1<p<\infty$. If $\Omega$ is a convex body in $\rnnn$ containing the origin and solving Problem 2 for $(\mu,p)$, then $\Omega$ is a convex proper polytope containing the origin in its interior.
\end{lem}
\begin{proof}
We do some preparation. Let
\begin{equation}\label{DE}
\mu=\sum_{i}^{m}c_{i}\delta_{\xi_{i}},
\end{equation}
where $\xi_{i}\in \sn$ are pairwise distinct unit vectors not contained in any closed hemisphere, and $c_{i}>0$ for all $i$. Then, for a convex body $Q$ containing the origin, we define the functional by
\begin{equation}\label{bv1}
F_{p}(Q)=\frac{1}{n+2}\sum^{m}_{i=1}c_{i}h(Q,\xi_{i})^{p}.
\end{equation}
For $y\in [0,\infty)^{m}$, define $P(y)$ the convex polytope by
\begin{equation}\label{bv2}
P(y)=\bigcap^{m}_{i=1}\{x\in \rnnn: x\cdot \xi_{i}\leq y_{i}\}.
\end{equation}
Now, we first imply that $\Omega$ is a convex proper polytope. Let
\begin{equation}\label{bv3}
h=(h_{1},\ldots,h_{m})=(h(\Omega,\xi_{1}),\ldots,h(\Omega,\xi_{m})).
\end{equation}
 Clearly, one see $\Omega\subseteq P(h)$, we have $T(\Omega)\leq T(P(h))$. Since $F_{p}(P(h))=F_{p}(\Omega)$, the proper polytope $P(h)$ satisfies the constraint in Problem 2. On the other hand, since the convex body $\Omega$ attacks Problem 2,  $T(\Omega)=T(P(h))$, it tells us that $P(h)$ also solves Problem 2. Together with Lemma \ref{LP33},  we get $P(h)=\Omega$.

 Next, we shall verify that $P(h)$ contains the origin in its interior. For this purpose, we take contradiction argument. Assuming that $o\in \partial P(h)$, we construct a new proper polytope $P(z)$ such that
 \begin{equation}\label{bv4}
o\in {\rm int} P(z), \quad F_{p}(P(z))\leq 1, \quad but \ T(P(z))> T(P(h)).
\end{equation}
 Since $o\in \partial P(y)$,  we assume that
 \begin{equation}\label{bv5}
h_{1}=\ldots=h_{k}=0,\quad and\quad h_{k+1},\ldots, h_{m}>0, \ for \ some \ 1\leq k<m.
\end{equation}
Now, let
\begin{equation}\label{bv6}
c=\frac{\sum^{k}_{i=1}c_{i}}{\sum^{m}_{i=k+1}c_{i}}
\end{equation}
and $t>0$ be sufficiently small, such that $0<t<t_{0}<\min\{ h^{p}_{i}/c:1\leq i \leq k\}^{1/p}$. Define
\begin{equation}\label{bv7}
y_{t}=(y_{1,t},\ldots,y_{m,t})=\left(t,\ldots,t,(h^{p}_{k+1}-ct^{p})^{1/p},\ldots,(h^{p}_{m}-ct^{p})^{1/p}\right).
\end{equation}
Hence $P(y_{0})=P(h)$, and $P(y_{t})$ is the convex polytope containing the origin in its interior.

Next, we give some observations.

First, we discover that $P(y_{t})$ is continuous, i.e.,
\begin{equation}\label{AP1}
\lim_{t\rightarrow 0^{+}}P(y_{t})=P(h).
\end{equation}

Second, in view of the fact that $P(y_{t})$ has at most $m$ facets whose outer unit normals are from the set $\{\xi_{1},\ldots,\xi_{m}\}$, moreover, $h(P(y_{t},\xi_{i}))\leq y_{i,t}$ with equality if $S(P(y_{t}),\xi_{i})>0$, for $i=1,\ldots,m$, combining the definition of $T$ with the absolute of $\mu^{tor}(P(y_{t}),\cdot)$ with regard to $S(P(y_{t}),\cdot)$, we get
\begin{equation}\label{AP2}
T(P(y_{t}))=\frac{1}{n+2}\sum^{m}_{i=1}y_{i,t}\mu^{tor}(P(y_{t}),\{\xi_{i}\}).
\end{equation}

Third, for $t_{1}, t_{2}\in [0,t_{0}]$, together with the definition of the mixed torsional rigidity, we have
\begin{equation}\label{AP3}
T(P(y_{t_{1}}), P(y_{t_{2}})))=\frac{1}{n+2}\sum^{m}_{i=1}h(P(y_{t_{2}},\xi_{i}))\mu^{tor}(P(y_{t_{1}}),\{\xi_{i}\}).
\end{equation}

Fourth, since there is at least one facet of $P(y_{0})$ containing $o$ such that $\sum^{k}_{i=1}S(P(y_{0}),\{\xi_{i}\})>0$, we obtain
\begin{equation}\label{AP4}
\sum^{k}_{i=1}\mu^{tor}(P(y_{0}),\{\xi_{i}\}))>0.
\end{equation}

Using \eqref{AP1}, \eqref{AP2}, \eqref{AP3}, \eqref{AP4} and the weak convergence of torsional measure, we have
\begin{equation}
\begin{split}
\label{II}
&(n+2)\lim_{t\rightarrow 0^{+}}\frac{T(P(y_{t}))-T(P(y_{t}),P(y_{0}))}{t}\\
&=\sum^{k}_{i=1}\lim_{t\rightarrow 0^{+}}\frac{t-0}{t}\mu^{tor}(P(y_{t}),\{\xi_{i}\})\\
&\quad +\sum^{m}_{i=k+1}\lim_{t\rightarrow 0^{+}}\frac{(h^{p}_{i}-ct^{p})^{1/p}-h_{i}}{t}\mu^{tor}(P(y_{t}),\{\xi_{i}\})\\
&=\sum^{k}_{i=1}\mu^{tor}(P(y_{0}),\{\xi_{i}\})\\
&>0.
\end{split}
\end{equation}
Applying \eqref{II} into the torsional Minkowski inequality, in conjunction with the continuity of $T(P(y_{t}))$ in $t$, we derive
\begin{equation*}
\begin{split}
\label{Up3}
&T(P(y_{0}))^{\frac{n+1}{n+2}}\liminf_{t\rightarrow 0^{+}}\frac{T(P(y_{t}))^{\frac{1}{n+2}}-T(P(y_{0}))^{\frac{1}{n+2}}}{t}\\
&=\liminf_{t\rightarrow 0^{+}} \frac{T(P(y_{t}))-T(P(y_{t}))^{\frac{n+1}{n+2}}T(P(y_{0}))^{\frac{1}{n+2}}}{t}\\
&\geq \liminf_{t\rightarrow 0^{+}}\frac{T(P(y_{t}))-T(P(y_{t}),P(y_{0}))}{t}\\
&>0.
\end{split}
\end{equation*}
So, we conclude $T(P(y_{t}))>T(P(y_{0}))$ for sufficiently small $t>0$.

We are in position to take a sufficiently small $t>0$ and let $z=y_{t}$. To reveal that $P(z)$ is indeed a convex polytope satisfying \eqref{bv6}, we need to claim $F(P(z))\leq 1$.

In fact, in light of \eqref{bv1} and $P(z)=P(y_{t})$, and $h(P(y_{t}),\xi_{i})\leq y_{i,t}$ for all $i$, \eqref{bv7} and $F_{p}(P(h))=\frac{1}{n+2}\sum^{m}_{k+1}c_{i}h^{p}_{i}$ (from \eqref{bv1}, \eqref{bv2}, \eqref{bv3}, \eqref{bv5}), \eqref{bv6} and the fact that $P(h)$ is a solution to Problem 2, it suffices to have
\begin{equation}
\begin{split}
\label{Up3}
F(P(z))&=\frac{1}{n+2}\sum^{m}_{i=1}c_{i}h(P(y_{t},\xi_{i})^{p}\\
&\leq \frac{1}{n+2}\sum^{m}_{i=1}c_{i}y^{p}_{i,t}\\
&=\frac{1}{n+2}\sum^{k}_{i=1}c_{i}t^{p}+\frac{1}{n+2}\sum^{m}_{i=k+1}c_{i}(h^{p}_{i}-ct^{p})\\
&=F_{p}(P(h))+\frac{1}{n+2}\left( \sum^{k}_{i=1}c_{i}-c\sum^{k}_{i+1}c_{i}\right)t^{p}\\
&=F_{p}(P(h))\\
&=1.
\end{split}
\end{equation}
Hence, we complete the proof.
\end{proof}

\emph{Proof of Theorem \ref{AD}.} Assume $\mu=\sum^{m}_{i=1}c_{i}\delta_{\xi_{i}}$, where $\xi_{i}\in \sn$ and $c_{i}>0$ for all $i$. Let $Q$ be convex body containing the origin, we set
\[
F_{p}(Q)=\frac{1}{n+2}\sum^{m}_{i=1}c_{i}h(Q,\xi_{i})^{p}.
\]
In view of the relations of solutions to Problems 1, 2 and 3, so, we only need treat that Problem 1 for $(\mu,p)$ has a solution $P_{0}$, where $P_{0}$ is a convex proper polytope containing the origin.

First, we aim to affirm that the minimizing sequence $\{P_{j}\}_{j}$ for Problem 1 shall satisfy that each $P_{j}$ is a convex proper polytope with facet normals locating in $\{\xi_{1},\ldots,\xi_{m}\}$. To realize this purpose, we set the Wulff shape related to each convex body $Q$ as
\[
P=\{x\in \rnnn: x\cdot \xi_{i}\leq h(Q,\xi_{i}),i=1,\ldots,m\}.
\]
In view of the fact that $\mu$ is not concentrated on any closed hemisphere, $Q$ is a convex body containing the origin and $P$ is a bounded convex polytope containing $Q$. Furthermore, $h(Q,\xi_{i})=h(P,\xi_{i})$ for all $i$. Hence, one see
\[
T(P)\geq T(Q), \quad F_{p}(P)=F_{p}(Q).
\]
So, we conclude that $\{P_{j}\}_{j}$ with faces orthogonal to $\xi_{i}$ ($i=1,\ldots,m$) is the desired minimizing sequence for Problem 1.

Second, we are devoted to get the boundedness of $\{P_{j}\}$. Since $T(T(B)^{-1/(n+2)}B)=1$ for $0<T(B)<\infty$, then
\[
\inf\{F_{p}(Q): o\in Q, \ T(Q)\geq 1\}\\
\leq M:=F_{p}(T(B)^{-1/(n+2)}B)< \infty.
\]
This implies  $F_{p}(P_{j})\leq M$ for all $j$. Since $c_{i}$ is positive and each $P_{j}$ contains the origin for all $i$ and $j$ respectively, we have
\[
\frac{1}{n+2}\min\{c_{i}:i=1,\ldots,m\}h(P_{j},\xi_{i})^{p}\leq \frac{1}{n+2}c_{i}h(P_{j},\xi_{i})^{p}\leq M.
\]
Consequently,
\[
h(P_{j},\xi_{i})\leq \left((n+2)\frac{M}{\min\{c_{i}:i=1,\ldots,m\}} \right)^{1/p}< \infty.
\]
Hence, the minimizing sequence $\{P_{j}\}_{j}$ is bounded from above. By the Blaschke selection theorem, it implies that $\{P_{j}\}_{j}$ has a convergent subsequence, still denoted by $\{P_{j}\}_{j}$, which converges to convex polytope $P_{0}$.

Third, we show dim $(P_{0})=n$. By lemma \ref{argu}, we know that
$Vol(P_{j})\geq \frac{1}{diam(P_{j})^{2}}T(P_{j})\geq \frac{1}{d^{2}}>0$ for some $d>0$. It follows that
\[
Vol(P_{0})=\lim_{j\rightarrow \infty}Vol(P_{j})>0.
\]
Hence, we get the desired result.

We are in place to return to Problem 2 for $(\mu,p)$. Since $P_{0}$ is a convex proper polytope dealing with Problem 1, making use of lemma \ref{re12}(1), we assert that $P:=F_{p}(P_{0})^{-1/p}P_{0}$ is a convex proper polytope dealing with Problem 2. By Lemmas \ref{LP33} and \ref{ra}, we say that such solution $P$ is unique and contains the origin in its interior.

Finally, we are ready to attack Problem (3). From Lemma \ref{LP23}(1), we know that the polytope $P$ is the unique solution to Problem 3. So the formula
\[
T(P)^{-1}d\mu^{tor}(P,\cdot)=h(P,\cdot)^{p-1}d\mu
\]
can be expressed as that in Theorem \ref{AD}.

Now, we prove Theorem \ref{TM} by adopting the approximation technique.

\emph{Proof of Theorem \ref{TM}}. By the proof of  Minkowski existence theorem showed in \cite{S14} or \cite{Hg05}, we know that, for the given measure $\mu$, there exists  a sequence of discrete measures $\{\mu_{j}\}$ defined on $\sn$
whose support is not contained in a closed hemisphere such that $\mu_{j}\rightarrow \mu$ weakly as $j\rightarrow \infty$. By Theorem \ref{AD}, there exists a polytope $P_{j}$ containing the origin with
\[
\mu_{j}=\frac{h(P_{j},\cdot)^{1-p}}{T(P_{j})}\mu^{tor}(P_{j},\cdot).
\]
In view of lemmas \ref{hj} and \ref{ui}, one see that the sequence $P_{j}$ is uniformly bounded, hence, by the Blaschke selection theorem, we assume the $P_{j}$ converges to a compact convex set $\Omega$ with $o\in \Omega$, using lemma \ref{hj} again, we know $T(P_{j})\geq c_{0}> 0$ for some $c_{0}>0$. By virtue of Lemma \ref{argu}, one see $Vol(P_{j})\geq \frac{1}{diam(P_{j})^{2}}T(P_{j})>0$, so $Vol(\Omega)=\lim_{j\rightarrow \infty}Vol(P_{j})>0$, thus $\Omega$ is a convex body containing the origin.

On the other hand, given a continuous function $f\in C(\sn)$, we get
\[
\int_{\sn}f(u)T(P_{j})h(P_{j},u)^{p-1}d\mu_{j}(u)=\int_{\sn}f(u)d\mu^{tor}(P_{j},u).
\]
Since $T(P_{j})h(P_{j},\cdot)^{p-1}\rightarrow T(\Omega)h(\Omega,\cdot)^{p-1}$ uniformly on $\sn$ for $p>1$, $\mu_{i}\rightarrow \mu$ and $\mu^{tor}(P_{j},\cdot)\rightarrow \mu^{tor}(\Omega,\cdot)$ weakly as $i\rightarrow \infty$. Hence, we derive that
\[
\int_{\sn}f(u)T(\Omega)h(\Omega,u)^{p-1}d\mu(u)=\int_{\sn}f(u)d\mu^{tor}(\Omega,u)
\]
for any $f\in C(\sn)$. This illustrates that Theorem \ref{TM} holds.

\section*{Acknowledgment}The authors would like to thank their supervisor  Prof. Yong Huang for valuable comments regarding the exposition of this paper.

\end{document}